\documentclass[10pt]{amsart}
\usepackage{amssymb,amsbsy,amsmath,amsfonts,times, amscd}
\usepackage{latexsym,euscript,exscale}

\usepackage{helvet}

\title{Non-unitarisable representations and maximal symmetry}
\author{Valentin Ferenczi}
\address{Instituto de Matem\'atica e Estat\'istica \\
 Universidade de S\~ao Paulo \\
rua do Mat\~ao 1010 \\
Cidade Universit\'aria \\
05508-90 S\~ao Paulo, SP \\
Brazil   \\ and \newline
Equipe d'Analyse Fonctionnelle \\
Institut de Math\'ematiques de Jussieu \\
Universit\'e Pierre et Marie Curie - Paris 6 \\
Case 247, 4 place Jussieu \\
75252 Paris Cedex 05 \\
France.}
\email{ferenczi@ime.usp.br}

\author {Christian Rosendal}
\address{Department of Mathematics, Statistics, and Computer Science (M/C 249)\\
University of Illinois at Chicago\\
851 S. Morgan St.\\
Chicago, IL 60607-7045\\
USA}
\email{rosendal.math@gmail.com}
\urladdr{http://homepages.math.uic.edu/$~$rosendal}

\thanks{Ferenczi was supported by FAPESP, grant 2013/11390-4. Rosendal was partially supported by a Simons Foundation Fellowship (Grant
\#229959) and also recognises support from the NSF (DMS 1201295).}

\date {}
\newcommand{\norm}[1]{\lVert#1\rVert}
\newcommand{\Norm}[1]{\big\lVert#1\big\rVert}

\newcommand{\triple}[1]{|\!|\!|#1|\!|\!|}

\newcommand {\1}{\mathbf 1}

\newcommand {\F}{\mathbb F}

\newcommand {\Z}{\mathbb Z}
\newcommand {\C}{\mathbb C}

\newcommand{\eps}{\epsilon}

\newcommand{\tom} {\emptyset}

\newcommand{\equi}{\Leftrightarrow}

\newcommand{\til}{\rightarrow}

\newcommand{\Lim}[1]{\mathop{\longrightarrow}\limits_{#1}}

\newcommand {\del}{ \; \big| \;}

\newcommand {\ku} {\mathcal}

\newcommand{\inv}{^{-1}}
\newcommand{\Id}{{\rm Id}}

\newcommand {\e} {\exists}
\renewcommand {\a} {\forall}

\newtheorem{thm}{Theorem}
\newtheorem{cor}[thm]{Corollary}
\newtheorem{lemma}[thm]{Lemma}
\newtheorem{prop} [thm] {Proposition}

\newtheorem{prob}[thm]{Problem}

\theoremstyle{definition}

\newtheorem{exa}[thm]{Example}

\begin{document}

\begin{abstract} 
We investigate questions of maximal symmetry in Banach spaces and the structure of certain bounded  non-unitarisable groups on Hilbert space. In particular, we provide structural information about bounded groups with an essentially unique invariant complemented subspace. This is subsequently combined with rigidity results for the unitary representation of ${\rm Aut}(T)$ on $\ell_2(T)$, where $T$ is the countably infinite regular tree, to describe the possible bounded subgroups of ${\rm GL}(\ku H)$ extending a well-known non-unitarisable representation of $\mathbb F_\infty$.

As a related result, we also show that a transitive norm on a separable Banach space must be strictly convex.  
\end{abstract}

\maketitle
\section{introduction}
The research of the present paper aims to expand on a circle of ideas involving maximal symmetry in Banach spaces and non-unitarisable representations in Hilbert space. Let us recall that a subgroup $G$ of the general linear group ${\rm GL}(X)$ of all continuous linear automorphisms of a Banach space $X$ is said to be {\em bounded} if $G$ is a uniformly bounded family of operators, i.e., $\sup_{T\in G}\norm T<\infty$. In this case, $X$ admits an equivalent $G$-invariant norm, namely, $\triple x=\sup_{T\in G}\norm {Tx}$. Thus, boundedness simply means that $G$ is a group of isometries for some equivalent norm on $X$. Also, $G\leqslant {\rm GL}(X)$ is {\em maximal bounded} if it is not properly contained in another bounded subgroup of ${\rm GL}(X)$. Maximal bounded groups naturally correspond to maximally symmetric norms on $X$, in the sense that, if the isometry group of a specific norm is maximal bounded, in which case we say the norm is {\em maximal}, then there is no manner of renorming $X$ to obtain a strictly larger set of isometries.  

When $X$ is finite-dimensional, every bounded $G\leqslant {\rm GL}(X)$ is contained in a maximal bounded subgroup, namely, the unitary group of a $G$-invariant inner product on $X$. This may be seen as an analogue of  the Cartan--Iwasawa--Malcev theorem, i.e.,  the existence of maximal compact subgroups in connected Lie groups. Also, in many of the classical spaces such as $\ell_p$ or Hilbert space, the canonical norm is maximal \cite{pelczynski, rolewicz}. However, not every Banach space admits an equivalent maximal norm, indeed, counter-examples may be found among super-reflexive spaces \cite{duke}. Furthermore, as shown by S. J. Dilworth and B. Randrianantoanina \cite{beata}, even among classical spaces such as $\ell_p$, $1<p<\infty$, $p\neq 2$, the general linear group contains bounded subgroups not contained in a maximal bounded subgroup. The following problem, which is the main motivation for our study, also remains stubbornly open.

\begin{prob}\label{prob}
Is every maximal norm on a Hilbert space $\ku H$ euclidean, i.e., generated by an inner product?
\end{prob}

This problem is tightly related to two other issues in functional analysis, namely, the existence of non-unitarisable bounded representations and S. Mazur's rotation problem. Here a bounded representation $\lambda\colon \Gamma\til {\rm GL}(\ku H)$ of a group $\Gamma$ on a complex Hilbert space $\ku H$ is said to be {\em unitarisable} if there is an equivalent $\lambda(\Gamma)$-invariant inner product on $\ku H$, or, equivalently, if $\lambda$ is conjugate to a unitary representation on $\ku H$. As was shown by  M. Day \cite{day} and J. Dixmier \cite{dixmier}, strongly continuous bounded representations of amenable groups are always unitarisable. On the other hand, L. Ehrenpreis and F. I. Mautner \cite{mautner} constructed the first example of a non-unitarisable bounded representation of a countable group $\Gamma$. In this connection, Dixmier posed the still central problem of whether unitarisability of all bounded representations characterises amenable groups among countable discrete groups. Now, by the Ehrenpreis--Mautner example, there are bounded subgroups $G\leqslant {\rm GL}(\ku H)$ of complex separable Hilbert space not preserving any euclidean norm, but it remains an open question whether there are such $G$ which are maximal or even if every such $G$ is contained in a maximal bounded subgroup.

Note that while the isometry group of a maximal norm on a space $X$ is maximal bounded by definition, there may be several essentially distinct norms, i.e., not scalar multiplies of each other, with this same isometry group. One case where the norm is uniquely defined by its isometry group is when the latter acts transitively on every sphere, in which case, the norm is said to be {\em transitive}. This happens, for example, for Hilbert space $\ku H$ and ultrapowers of $L^p$ spaces. However, in the separable setting, it is not known whether Hilbert space is the only such example either isomorphically or isometrically. This is known as {\rm Mazur's rotation problem} \cite{Banach,mazur}. 

Though much information has been obtained on {\em almost transitive} Banach spaces, i.e., whose isometry group has dense orbits on spheres, under additional geometric assumptions such as reflexivity \cite{CS, becerra}, we are not aware of any results that necessitates actual transitivity. Our first result shows that such spaces are at least strictly convex.

\begin{thm}
Let $(X,\norm\cdot)$ be a separable real transitive Banach space. Then $X$ is strictly convex and $\norm\cdot$ is G\^ateaux differentiable.
\end{thm}

Let us remark that this result fails if $X$ is only assumed to be almost transitive, as can be seen by considering $L^1([0,1])$.

We then turn our attention to issues related to Problem \ref{prob}. In particular, we shall be considering the structure of bounded groups containing the image of one specific widely studied non-unitarisable representation associated to actions on trees (see, e.g., \cite{ ozawa, pisier, pytlic}). For this, suppose that $\lambda\colon \Gamma\til \ku U(\ku H)$  is a unitary presentation. A {\em bounded derivation} associated to $\lambda$ is a uniformly bounded map $d\colon \Gamma\til \ku B(\ku H)$ so that $d(gf)=\lambda(g)d(f)+d(g)\lambda(f)$ for all $g,f\in \Gamma$. This is simply equivalent to requiring that
$$
\lambda_d(g)=\begin{pmatrix} \lambda(g) & d(g) \\ 0 & \lambda(g) \end{pmatrix}
$$
defines a bounded represention of $\Gamma$ on $\ku H\oplus \ku H$. It is well-known that the representation $\lambda_d$ is unitarisable exactly when $d$ is {\em inner}, i.e., $d(g)=\lambda(g)A-A\lambda(g)$ for some bounded linear operator $A$ on $\ku H$.

With the aim of elucidating bounded groups $G\leqslant {\rm GL}(\ku H\oplus \ku H)$ containing $\lambda_d[\Gamma]$ for $\lambda$ and $d$ as above, which are potential examples of maximal non-unitarisable groups, we first prove a result valid in a much broader setting relating the diagonal entries in a bounded group of upper triangular block matrices. Indeed, suppose that $X=Y\oplus Z$ is a separable reflexive Banach space and $G\leqslant {\rm GL}(Y\oplus Z)$ is a bounded group of upper triangular block matrices 
$$
\begin{pmatrix} u & w \\ 0 & v \end{pmatrix}.
$$
We first observe that, in this case, $w$ is actually a function $w$ of $u,v$. However, under stronger assumptions, we show that  the diagonal entries $u$ and $v$ are also in a one-to-one correspondence and so every element of $G$ is uniquely determined by just the entry $u$ and similarly by $v$.

\begin{thm}
Let $X=Y\oplus Z$ be separable reflexive and $G\leqslant {\rm GL}(X)$ a bounded subgroup leaving $Y$ invariant. Assume that there are no closed linear $G$-invariant subspaces $\{0\}\varsubsetneq W\varsubsetneq Y$ nor superspaces $Y\varsubsetneq W\varsubsetneq X$ and there is no closed linear $G$-invariant complement of $Y$ in $X$.
Then the mappings
$$
\begin{pmatrix} u & w \\ 0 & v \end{pmatrix}\mapsto u
\quad\text{and}\quad
\begin{pmatrix} u & w \\ 0 & v \end{pmatrix}\mapsto v
$$
are $\mathtt{sot}$-isomorphisms between $G$ and the respective images in ${\rm GL}(Y)$ and ${\rm GL}(Z)$.
\end{thm}

This result in particular applies when $G$ contains the image $\lambda_d[\Gamma]$, where $\lambda$ is an irreducible unitary representation and $d$ is an associated non-inner derivation, whence the following corollary.

\begin{cor}
Suppose that $\lambda\colon \Gamma\til \ku U(\ku H)$ is an irreducible unitary representation of a group $\Gamma$ on a separable  Hilbert space $\ku H$ and $d\colon \Gamma\til \ku B(\ku H)$ is an associated non-inner bounded derivation.
Suppose that $G \leqslant GL(\ku H \oplus \ku H)$ is a bounded subgroup leaving the first copy of $\ku H$ invariant and containing $\lambda_d[\Gamma]$. Then the mappings $G\til {\rm GL}(\ku H)$ defined by
$$
\begin{pmatrix} u & w \\ 0 & v \end{pmatrix}\mapsto u
\quad\text{and}\quad
\begin{pmatrix} u & w \\ 0 & v \end{pmatrix}\mapsto v
$$
are $\mathtt{sot}$-isomorphisms between $G$ and the respective images in ${\rm GL}(\ku H)$.
\end{cor}

We subsequently turn to study one specific derivation. For this, let $T$ denote the $\aleph_0$-regular tree, i.e., the Cayley graph of the free group $\F_\infty$ on denumerably many generators with respect to its free generating set. Let also ${\rm Aut}(T)$ denote its group of automorphisms and $\lambda\colon {\rm Aut}(T)\curvearrowright \C^T$ the canonical shift action on the vector space of $\C$-valued functions on $T$. We observe that each of the subspaces $\ell_p(T)\subseteq \C^T$ are $\lambda$-invariant. While it is fairly easy to see that the unitary representation $\lambda\colon {\rm Aut}(T)\til \ku U(\ell_2(T))$ is irreducible and  {\em uniquely unitarisable}, i.e., up to a scalar multiple preserves a unique inner product equivalent with the usual one, we may show significantly stronger results. Firstly, we show that the usual inner prooduct $\langle\cdot|\cdot\rangle$, up to a scalar multiple, is the only inner product (not necessarily equivalent to $\langle\cdot|\cdot\rangle$) preserved by $\lambda$. Secondly, we have the following result strengthening irreducibility.

\begin{thm}\label{commutant intro}
The commutant of $\lambda\big({\rm Aut}(T)\big)$ in the space of linear operators from $\ell_p(T)$ to $\C^T$, $1<p\leqslant \infty$, is just $\C\!\cdot\! \Id$. 
\end{thm}

To construct the derivation, we fix a root $e\in T$ and set $\hat e=e$, while for $s\in T$, $s\neq e$, we let $\hat s$ denote the penultimate vertex on the geodesic in $T$ from $e$ to $s$. Also,  $L\colon \ell_1(T)\til \ell_1(T)$ is the bounded linear operator satisfying $L(\1_s)=\1_{\hat s}$ for $s\neq e$ and $L(\1_e)=0$. Then, if $L^*$ denotes the adjoint operator on $\ell_\infty(T)$, for every $g\in {\rm Aut}(T)$, $d(g)=L^*\lambda(g)-\lambda(g)L^*$ restricts to a linear operator on $\ell_2(T)$ of norm $\leqslant 2$ and agrees with the operator $\lambda(g)L-L\lambda(g)$ on $\ell_1(T)$. It follows that $d$ defines a bounded derivation associated to $\lambda$, which, however, is not inner. Moreover, with the aid of Theorem \ref{commutant intro}, we show that this definition of $d$ is extremely rigid.  In fact, $L^*$ is essentially the only linear map with domain $\ell_2(T)$ defining $d$.

\begin{thm}Let $d$ be the derivation defined above and suppose $A\colon \ell_2(T)\til \C^T$ is a globally defined linear operator so that $d(g)=A\lambda(g)-\lambda(g)A$ for all $g\in {\rm Aut}(T)$. Then  $A=L^*+\vartheta\Id$ for some $\vartheta\in \C$. 
\end{thm}

Finally, we may combine the previous analysis of bounded subgroups with the specific nature of the given derivation to obtain the following structure result.

\begin{thm}
Let $d$ be the derivation defined above and suppose that $G\leqslant {\rm GL}(\ell_2(T)\oplus\ell_2(T))$ is a bounded subgroup leaving  the first copy of $\ell_2(T)$ invariant and containing $\lambda_d[{\rm Aut}(T)]$.
Then there is a continuous homogeneous map $\psi\colon \ell_2(T)\til \ell_2(T)$ for which
$$
L^*+\psi\colon \ell_2(T)\til \ell_\infty(T)\quad \text{and}\quad  L-\psi\colon \ell_1(T)\til \ell_2(T)
$$
commute with $\lambda(g)$ for $g\in {\rm Aut}(T)$ and so that every element of $G$ is of the form
$$
\begin{pmatrix} u & u\psi-\psi v \\ 0 & v \end{pmatrix}
$$
for some $u,v\in {\rm GL}(\ell_2(T))$. 

Finally, the mappings 
$$
\begin{pmatrix} u & u\psi-\psi v \\ 0 & v \end{pmatrix}\mapsto u
\qquad \text{and}\qquad
\begin{pmatrix} u &u\psi-\psi v \\ 0 & v \end{pmatrix}\mapsto v
$$
are $\tt{sot}$-isomorphisms between $G$ and their respective images in ${\rm GL}(\ell_2(T))$.
\end{thm}

We subsequently use this result to compute some simple values of $\psi$, which could be useful for extracting information about $G$.


\section{Strict convexity of separable transitive Banach spaces}
Let $(X,\norm\cdot)$ be a fixed separable {\em transitive} Banach space, i.e., whose linear isometry group ${\rm Isom}(X,\norm\cdot)$ acts transtively on the unit sphere $S_X=\{x\in X\del \norm x=1\}$. By a classical theorem of S. Mazur (Theorem 8.2 \cite{fabian}) the norm $\norm\cdot$ is G\^ateaux differentiable on a dense $G_\delta$ subset of $S_X$. So, by transitivity of the norm, this implies that $\norm\cdot$ is actually G\^ateaux differentiable at every point of $S_X$. Hence the following lemma.

\begin{lemma}\label{gateaux}
Let $(X,\norm\cdot)$ be a separable transitive Banach space. Then $\norm\cdot$ is G\^ateaux differentiable, i.e., for every $x\in S_X$, there is a unique support functional $\phi_x\in S_{X^*}$, that is, so that $\phi_x(x)=1$.
\end{lemma}

Now, G\^ateaux differentiablity of norms and strict convexity are related via duality (see Corollary 7.23 \cite{fabian}), in the sense that, e.g., G\^ateaux differentiability of the dual norm $\norm\cdot^*$ implying strict
convexity of $\norm\cdot$. However, little information can be gained directly from the G\^ateaux differentiability of the norm on $X$. Nevertheless, using the Bishop--Phelps theorem, we shall see that the norm is actually strictly convex. For this, let us recall the statement of the Bishop--Phelps theorem: If $C$ is a non-empty bounded closed convex subset of a real Banach space $X$, then the set
$$
\{\phi\in X^*\del \e x\in C\; \sup_{y\in C}\phi(y)=\phi(x)\}
$$
is norm-dense in $X^*$.

\begin{lemma}
Let $X$ be a separable real Banach space and $C\subseteq S_X$ a non-empty closed convex set so that the setwise stabiliser $\{T\in {\rm Isom}(X)\del T[C]=C\}$ acts transitively on $C$. Then $C$ consists of a single point. 
\end{lemma}

\begin{proof}
Since $X$ is separable, we can pick a dense sequence $(x_n)$ in $C$. Let also $\lambda_n>0$ be so that $\sum_n\lambda_n=1$ and define $x=\sum_n\lambda_nx_n\in C$ (note that since $\norm{x_n}=1$ the sum is absolutely convergent).  

We claim that, if $\phi\in X^*$ attains its maximum on $C$ at $x$, then $\phi$ must be constant on $C$. Indeed, in this case, 
$$
\phi(x)=\phi\big(\sum_n\lambda_nx_n\big)=\sum_n\lambda_n\phi(x_n)\leqslant \sum_n\lambda_n\phi(x)=\phi(x),
$$
so $\phi(x_n)=\phi(x)$ for all $n$, whence $\phi\equiv \phi(x)$ on $C$.

Now suppose for a contradiction that $C$ contains distinct points $y$ and $z$. We pick $\psi \in X^*$ of norm $1$ so that $\psi(y-z)=\eps>0$. Then, by the theorem of Bishop--Phelps, there is some $\phi\in X^*$ with 
$\norm{\psi-\phi}<\frac\eps{2\norm{y-z}}$ that attains it supremum on $C$ at some point $v\in C$. Also,
$$
|\phi(y-z)|\geqslant |\psi(y-z)|-|(\psi-\phi)(y-z)|\geqslant \eps-\norm{\psi-\phi}\cdot\norm{y-z}>\frac \eps2>0,
$$
so $\phi$ is not constant on $C$. 

Choose some $T\in {\rm Isom}(X)$ with $T[C]=C$ so that $Tx=v$ and note that $T^*\phi$ attains it maximum on $C$ at $x$ and thus must be constant on $C$. However, this is absurd, since $T[C]=C$ and $\phi$ fails to be constant on $C$.
\end{proof}

\begin{thm}
Let $(X,\norm\cdot)$ be a separable real transitive Banach space. Then $X$ is strictly convex and $\norm\cdot$ is G\^ateaux differentiable.
\end{thm}

\begin{proof}
We already know that $\norm\cdot$ is G\^ateaux differentiable and thus every $x\in S_X$ has a unique support functional $\phi_x\in S_{X^*}$. Now, for $x\in S_X$, consider the closed convex subset
$$
C_x=\{z\in S_X\del \phi_z=\phi_x\}=\{z\in S_X\del \phi_x(z)=1\},
$$
where the second equality follows from the uniqueness of the support functional. 

Then, for all $T\in {\rm Isom}(X)$, either $T[C_x]=C_x$ or $T[C_x]\cap C_x=\tom$. For, suppose $z,Tz\in C_x$.
Then
$T^*\phi_x(z)=\phi_x(Tz)=1=\phi_x(z)$, whence by uniqueness of support functionals we have $T^*\phi_x=\phi_x$ and hence
$$
\phi_x(Ty)=T^*\phi_x(y)=1
$$
for all $y\in C_x$, i.e., $T[C_x]\subseteq C_x$. Similarly, $T\inv [C_x]\subseteq C_x$ and thus $T[C_x]=C_x$. 

Therefore, as ${\rm Isom}(X)$ acts transitively on $S_X$, we see that $\{T\in {\rm Isom}(X)\del T[C_x]=C_x\}$ acts transitively on $C_x$ and hence, by the preceding lemma, $C_x=\{x\}$. It follows that the mapping $x\in S_x\mapsto \phi_x\in S_{X^*}$ is injective and that every functional $\phi\in S_{X^*}$ attains its norm in at most one point of $S_X$, that is, $X$ is strictly convex.
\end{proof}

We note that, by the theorem of Bishop-Phelps, the set of norm attaining functionals is norm dense in $X^*$, which, in our setting means that the $\phi_x$ for $x\in S_X$ are norm dense in $S_{X^*}$. Moreover, by the G\^ateaux differentiability of the norm, the mapping $x\in S_X\mapsto \phi_x\in S_{X^*}$ is $\norm\cdot$-to-$w^*$ continuous. The action of ${\rm Isom}(X)$ on $S_{X^*}$ is transitive on the set $\{\phi_x\}_{x\in S_X}$ and thus $X^*$ is almost transitive. However, little geometric information can be obtained exclusively from almost transitivity, since it is known by a result of W. Lusky \cite{lusky} that every Banach space $X$ is isometric to a complemented subspace of an almost transitive Banach space.


\section{On bounded representations with invariant subspaces}\label{bounded}
In the following, we consider a separable reflexive Banach space $X$ and a bounded subgroup $G\leqslant {\rm GL}(X)$ along with a $G$-invariant closed linear subspace $Y\subseteq X$.
We let $\pi\colon X\til X/Y$ denote the canonical quotient map and write $\dot x$ for $\pi(x)=x+Y\in X/Y$. Note also that every $T\in G$ induces an operator $\dot T\in {\rm GL}(X/Y)$ defined by
$$
\dot T(\dot x)=\big(Tx\big)^{\!\centerdot},
$$
i.e., $\dot T(x+Y)=Tx+Y$. Moreover, as $\norm{\dot T(\dot x)}=\Norm{\big(Tx\big)^{\!\centerdot}}\leqslant \norm{Tx}\leqslant \norm T\cdot\norm x$ for all $x\in X$, we see that $\norm{\dot T}\leqslant \norm T$. In particular, $\dot G=\{\dot T\in {\rm GL}(X/Y)\del T\in G\}$ is a bounded subgroup of ${\rm GL}(X/Y)$.

\begin{lemma}\label{lifting}
Let $X$ be a separable reflexive Banach space, $G\leqslant {\rm GL}(X)$  a bounded subgroup and 
suppose that $Y\subseteq X$ is a $G$-invariant closed linear subspace.
Then there is a continuous, homogeneous and thus bounded $G$-equivariant lifting $\phi\colon X/Y\til X$ of the quotient map $\pi$, that is, $\pi\circ \phi={\rm Id}_{X/Y}$ and $\phi\dot T=T\phi$ for all $T\in G$.
\end{lemma}

\begin{proof}
First, by results of G. Lancien \cite{lancien}, since $X$ is separable reflexive and $G\leqslant {\rm GL}(X)$ is bounded, there is an equivalent $G$-invariant LUR norm $\norm\cdot$ on $X$.
  In other words, $G$ is a subgroup of the group ${\rm Isom}(X,\norm\cdot)$ of linear isometries of $X, \norm\cdot$. 

Now, since $X$ is reflexive and $\norm\cdot$ is LUR, the map $c\colon X\til Y$ defined by
$$
c(x)=\text{ the unique point $y\in Y$ closest to $x$}
$$
is well-defined and continuous (see Exercise 7.47 \cite{fabian}). Note then that, for every $x\in X$,  $x-c(x)$ is the unique point in $x+Y$ of minimal norm. Let $b\colon X/Y\til X$ be a Bartle--Graves selector (see Corollary 7.56 \cite{fabian}), that is, $b$ is a continuous lifting of the quotient mapping $\pi\colon X\til X/Y$. We then let $\phi\colon X/Y\til X$
be defined by $\phi(z)=b(z)-c(b(z))$ and note that, for $x\in X$,  $\phi(\dot x)$ is the unique point  of minimal norm in the affine subspace $x+Y\subseteq X$. 

Suppose that $x\in X$ and $T\in G$. Then, since $T[Y]=Y$ and $T$ is a linear isometry of $X$, 
\[\begin{split}
\phi(\dot T \dot x)
&=\phi\big((Tx)^{\centerdot}\big)\\
&=\text{ the unique point in $Tx+Y$ of minimal norm}\\
&=\text{ the unique point in $T[x+Y]$ of minimal norm}\\
&=T\big(\text{the unique point in $x+Y$ of minimal norm}\big)\\
&=T\phi(\dot x).
\end{split}\]
Thus, $\phi\dot T=T\phi$ for all $T\in G$, i.e., $\phi$ is a continuous $G$-equivariant lifting of the quotient map. Similarly, for $x\in X$ and  $\lambda$ a scalar, 
\[\begin{split}
\phi(\dot x)=x
&\equi\a y\in Y\; \norm{x}\leqslant \norm{x+y}\\
&\equi\a y\in Y\; \norm{\lambda x}\leqslant \norm{\lambda (x+y)}\\
&\equi\a y\in Y\; \norm{\lambda x}\leqslant \norm{\lambda x+y}\\
&\equi\phi(\lambda\dot x)=\lambda x,
\end{split}\]
whence $\phi$ is homogeneous. 
Finally, let us also note that $\|\phi(\dot x)\|=\|\dot x\|$ for all $x \in X$. 
\end{proof}

We observe that in Lemma \ref{lifting}, if $\norm \cdot$ denotes the original norm on $X$, then $\phi$ may be chosen of norm at most $(1+\epsilon) \sup_{T \in G}\norm T$, for any choice of $\epsilon>0$. This easily follows from the construction of a $G$-invariant LUR norm on $X$. Indeed,   define $\triple{x}=\sup_{T \in G}\norm{Tx}$, and  let $\triple\cdot'$ be an equivalent LUR norm
for which  
$$
{\rm Isom}(X, \triple \cdot) \leqslant {\rm Isom}(X, \triple \cdot').
$$ 
By density of the LUR property in the space of norms (\cite{DGZ} or Proposition 4.5 \cite{duke}), $\triple \cdot'$ may be chosen so that
 $\triple \cdot \leqslant {\triple \cdot'} \leqslant (1+\epsilon)\triple \cdot$.  Letting $\phi$ denote the lifting of Lemma \ref{lifting}, we have $\triple{\phi(\dot x)}'=\triple{\dot x}'$ and thus
$$
\norm{\phi (\dot x)} \leqslant \triple{\phi (\dot x)} \leqslant \triple{\phi (\dot x)}'=\triple{\dot x}' \leqslant (1+\epsilon) \triple{\dot x} \leqslant
(1+\epsilon) \big(\sup_{T \in G}\norm{T}\big) \norm{\dot x}.
$$
for all $x\in X$.
This estimate may be improved by dropping the continuity property.

\begin{lemma}  
Let $X$ be a separable reflexive Banach space, $G\leqslant {\rm GL}(X)$  a bounded subgroup and 
suppose that $Y\subseteq X$ is a $G$-invariant closed linear subspace. Then there exists a homogeneous $G$-equivariant lifting $\phi\colon X/Y\til X$ of the quotient map $\pi$ with norm at most $\sup_{T \in G}\|T\|$.
\end{lemma}

\begin{proof}
By the above remark, let $\phi_n$ be a homogeneous $G$-equivariant lifting
associated to a choice of equivalent $G$-invariant LUR renorming ${\norm\cdot}_n$ of norm
at most $(1+\frac{1}{n}) \sup_{T \in G}\|T\|$. 
Use reflexivity to define, for all $\dot x \in X/Y$, $\phi(\dot x)$ as a weak limit along a non-trivial ultrafilter,
$$
\phi(\dot x)=w-\lim_{n\til\mathcal U} \phi_n(\dot x).
$$
It is easily checked that $\phi$ is a  homogeneous $G$-equivariant lifting of $\pi$ of norm
at most  $\sup_{T \in G}\|T\|$. 
\end{proof}

Note that, if $\phi$ is the lifting defined by Lemma \ref{lifting}, then $p\colon X\til Y$ given by $p(x)=x-\phi(\dot x)$ is a continuous homogeneous (potentially non-linear) projection of $X$ onto its subspace $Y$. Moreover, in this case, we can define a homogeneous homeomorphism between $X$ and $Y\oplus X/Y$ via $x\mapsto (p(x), \dot x)$ with homogeneous inverse $(y,z)\mapsto y+\phi(z)$. By the $G$-equivariance of $\phi$, we also have
$$
Tx\mapsto (Tx-\phi(\dot T\dot x), \dot T\dot x)=(Tx-T\phi(\dot x), \dot T\dot x)=\big((T|_Y)(p(x)), \dot T\dot x\big),
$$
which shows that the action of $G$ on $X$ is conjugate by the above homeomorphism to the $G$-action on $Y\oplus X/Y$ given by the block diagonal representation
$$
T\mapsto \begin{pmatrix} T|_Y & 0 \\ 0 & \dot T \end{pmatrix}.
$$

\begin{lemma}\label{existence of derivation} Let $X$ be a separable reflexive Banach space, $G\leqslant {\rm GL}(X)$  a bounded subgroup and 
suppose that $Y\subseteq X$ is a $G$-invariant closed linear subspace.
Then the mapping $T \mapsto \begin{pmatrix} T|_Y & 0 \\ 0 & \dot T \end{pmatrix}$
is an injective homomorphism of $G$ into ${\rm GL}(Y \oplus X/Y)$.
\end{lemma}

\begin{proof} Assume $T \in G$ satisfies $T|_Y=\Id_Y$ and $\dot T=\Id_{X/Y}$.
Then $T$ acts as the identity on $Y \oplus X/Y$, and
since the action of $G$ on $X$ is conjugate by homeomorphism to the action
of $G$ on $Y \oplus X/Y$, we deduce that $T=\Id$.
\end{proof}

\begin{thm}\label{sot isom}
Let $X$ be a separable reflexive Banach space and $G\leqslant {\rm GL}(X)$ be  a bounded subgroup. 
Suppose that $Y\subseteq X$ is a $G$-invariant closed linear subspace so that
\begin{itemize}
\item[(i)] there is no closed linear $G$-invariant subspace $\{0\}\varsubsetneq W\varsubsetneq Y$, 
\item[(ii)] there is no closed linear $G$-invariant complement of $Y$ in $X$.
\end{itemize}
Then the mapping $T\mapsto \dot{T}$ is an isomorphism of the topological groups $(G,\mathtt{sot})$ and $(\dot G, \mathtt{sot})$. 
\end{thm}

\begin{proof}
Let $\phi\colon X/Y\til X$ be the lifting given by Lemma \ref{lifting}.
Define $\Delta\colon X/Y\times X/Y\til X$ by $\Delta(\dot x_1,\dot x_2)=\phi(\dot x_1)+\phi(\dot x_2)-\phi(\dot x_1+\dot x_2)$ and observe that, since $ \phi(\dot x_1)\in x_1+Y$, $\phi(\dot x_2)\in x_2+Y$ and $\phi(\dot x_1+\dot x_2)\in (x_1+x_2)+Y$, we have $\Delta(\dot x_1,\dot x_2)\in Y$. Moreover, by $G$-equivariance of $\phi$, we find that
\[\begin{split}
T\Delta(\dot x_1,\dot x_2)
&=T\phi(\dot x_1)+T\phi(\dot x_2)-T\phi(\dot x_1+\dot x_2)\\
&=\phi(\dot T\dot x_1)+\phi(\dot T\dot x_2)-\phi(\dot T\dot x_1+\dot T\dot x_2)\\
&=\Delta(\dot T\dot x_1,\dot T\dot x_2)
\end{split}\]
for all $T\in G$ and $x_1,x_2\in X$.

We claim that $\Delta(\dot x_1,\dot x_2)\neq 0$ for some $x_1,x_2\in X$. Indeed, suppose not. Then $\phi$ is a bounded linear $G$-equivariant map, whereby the composition $P=\phi\circ\pi$ is a bounded linear projection with $\ker P=Y$ satisfying 
$$
PT(x)=\phi\pi (Tx)=\phi\big(\dot T\dot x)=T\phi(\dot x)=TP(x)
$$
for all $x\in X$, i.e., $PT=TP$. So $W=P[X]$ is a $G$-invariant closed linear complement of $Y$ in $X$, contradicting our assumption.

Thus, as $0\neq \Delta(\dot x_1,\dot x_2)\in Y$ and there are no non-trivial $G$-invariant closed linear subspaces of $Y$, we see that $\overline{\rm span}\big(G\cdot \Delta(\dot x_1,\dot x_2)\big)=Y$. 

We claim that, for all $T_n, T\in G$, we have 
$$
T_n\Lim{\mathtt{sot}}T
\;\;\equi \;\;
\dot T_n\Lim{\mathtt{sot}}\dot T.
$$
The implication from left to right is obvious. For the other direction, assume that $\dot T_n\Lim{\mathtt{sot}}\dot T$. Suppose first that $S\in G$ is given. Then, since $\phi$ and hence also $\Delta$ are continuous, we have  
$$
\lim_nT_nS\Delta(\dot x_1, \dot x_2)=\lim_n\Delta(\dot T_n\dot S \dot x_1, \dot T_n\dot S\dot x_2)=\Delta(\dot T\dot S \dot x_1, \dot T\dot S\dot x_2)=TS\Delta(\dot x_1, \dot x_2).
$$
As $\overline{\rm span}\big(G\cdot \Delta(\dot x_1,\dot x_2)\big)=Y$ and $G$ is a group of isometries, this shows that $T_ny\til Ty$ for all $y\in Y$.

Let now $x\in X$ be given and write $x=\phi(\dot x)+y$ for some $y\in Y$. Then, since $\phi$ is  continuous and $G$-equivariant, we have 
$$
T_nx=T_n\phi(\dot x)+T_ny=\phi(\dot T_n\dot x)+T_ny\;\Lim{n}\; \phi(\dot T\dot x)+Ty=T\phi(\dot x)+Ty=Tx,
$$
which shows that $T_n\Lim{\mathtt{sot}}T$. 

Since $T\til \dot T$ is clearly a group homomorphism, this shows that it is an isomorphism of the topological groups $(G, \mathtt{sot})$ and $(\dot G,\mathtt{sot})$. 
\end{proof}

\begin{cor}\label{isom general setting}
Let $X$ be a separable reflexive Banach space and $G\leqslant {\rm GL}(X)$ be  a bounded subgroup. 
Suppose that $Y\subseteq X$ is a $G$-invariant closed linear subspace so that
\begin{itemize}
\item[(i)] there is no closed linear $G$-invariant superspace $Y\varsubsetneq W\varsubsetneq X$, 
\item[(ii)] there is no closed linear $G$-invariant complement of $Y$ in $X$.
\end{itemize}
Then the mapping $T\mapsto T|_Y$  is an isomorphism of the topological groups $(G,\mathtt{sot})$ and $(G|_Y,\mathtt{sot})$, where $G|_Y=\{T|_Y\in {\rm GL}(Y)\del T\in G\}$.
\end{cor}

\begin{proof}As in the proof of Theorem \ref{sot isom}, we may assume that $G$ is a group of isometries of $X$.

Note that the short exact sequence
$$
0\til Y\til X\til X/Y\til 0
$$
gives rise to the short exact sequence 
$$
0\til Y^\perp \til X^*\til X^*/Y^\perp\til 0
$$
by duality. 

We set $G^*=\{T^*\in {\rm GL}(X^*)\del T \in G\}$. Then any $G^*$-invariant subspace $\{0\}\varsubsetneq V\varsubsetneq Y^\perp$ would induce a $G$-invariant subspace $Y\varsubsetneq W\varsubsetneq X$ by $W=V_\perp$. Similarly, a $G^*$-invariant complement $V$ of $Y^\perp$ in $X^*$ would induce a $G$-invariant complement $W=V_\perp$ of $Y$ in $X$. Therefore, we see that $G^*$ and $X^*$ satisfy the 
conditions of Theorem \ref{sot isom}, which means that the map $T^*\in G^*\mapsto (T^*)^\centerdot\in {\rm Isom}(X^*/Y^\perp)$ is an isomorphism of the topological group $(G^*,\mathtt{sot})$ with its image in   $\big({\rm Isom}(X^*/Y^\perp), \mathtt{sot}\big)$.

Now, since $X$ is reflexive, the map $T\mapsto T^*$ is an isomorphism of $\big({\rm Isom}(X), \mathtt{sot}\big)$ with $\big({\rm Isom}(X^*), \mathtt{sot}\big)$. Similarly, as $X^*/Y^\perp$ can be identified with $Y^*$ via Hahn--Banach, we again have an isomorphism between $\big({\rm Isom}(X^*/Y^\perp), \mathtt{sot}\big)$ and $\big({\rm Isom}(Y), \mathtt{sot}\big)$. Moreover, the composition of these three maps shows that $T\mapsto T|_Y$ is an isomorphism between $(G, \mathtt{sot})$ and $(G|_Y,\mathtt{sot})$. 
\end{proof}

\begin{cor}
Let $X$ be a separable reflexive Banach space and $G\leqslant {\rm GL}(X)$ be  a bounded subgroup. 
Suppose that $Y\subseteq X$ is a $G$-invariant closed linear subspace so that
\begin{itemize}
\item[(i)] there are no closed linear $G$-invariant subspaces $\{0\}\varsubsetneq W\varsubsetneq Y$ nor superspaces $Y\varsubsetneq W\varsubsetneq X$, 
\item[(ii)] there is no closed linear $G$-invariant complement of $Y$ in $X$.
\end{itemize}
Then the mapping $T|_Y\mapsto \dot T$ is well-defined and provides an isomorphism between the topological groups  
$(G|_Y,\mathtt{sot})$ and $(\dot G,\mathtt{sot})$.
\end{cor}

Now, returning to our original assumptions, we suppose that $X$ is a separable reflexive Banach space and $G\leqslant {\rm GL}(X)$ is a bounded subgroup preserving a closed linear subspace $Y\subseteq X$. Suppose furthermore that $Y$ is complemented in $X$, i.e., that we may write $X=Y\oplus Z$ for some closed linear subspace $Z\subseteq X$. Since $Y$ is $G$-invariant, with respect to the decomposition $X=Y\oplus Z$, every element $T\in G$ may be represented by a block matrix
$$
\begin{pmatrix} u_T & w_T \\ 0 & v_T \end{pmatrix},
$$
where $u_T\in {\rm GL}(Y)$, $v_T\in {\rm GL}(Z)$ and $w_T$ is a bounded linear operator from $Z$ to $Y$. Also, $u_T$ is simply the restriction $T|_Y$. Moreover, the projection map $\pi\colon X\til X/Y$ restricts to an isomorphism between $Z$ and $X/Y$ and we note that the operator $\dot T\in {\rm GL}(X/Y)$ is conjugate to $v_T$ via this isomorphism. So henceforth, we shall simply identify $X/Y$ with $Z$ and $\dot T$ with $v_T$.

By Lemma \ref{existence of derivation}, every element of $G$ is represented by a block matrix
$$
\begin{pmatrix} u & \delta(u,v) \\ 0 & v \end{pmatrix},
$$
where $\delta(u,v)\colon Z\til Y$ is a bounded linear operator uniquely determined as a function of $u$ and $v$. As
$$
\begin{pmatrix} u_1 & \delta(u_1,v_1) \\ 0 & v_1 \end{pmatrix}
\begin{pmatrix} u_2 & \delta(u_2,v_2) \\ 0 & v_2 \end{pmatrix}
=\begin{pmatrix} u_1u_2 & u_1\delta(u_2,v_2)+\delta(u_1,v_1)v_2 \\ 0 & v_1v_2 \end{pmatrix},
$$
we find that 
$$
\delta(u_1u_2,v_1v_2)=u_1\delta(u_2,v_2)+\delta(u_1,v_1)v_2.
$$

\begin{lemma}\label{repofderivation}
Let $X=Y\oplus Z$ be separable reflexive and $G\leqslant {\rm GL}(X)$ a bounded subgroup leaving $Y$ invariant. Then there is a continuous homogeneous map $\psi\colon Z\til Y$ so that
$$
\delta(u,v)=u\psi-\psi v.
$$
\end{lemma}

\begin{proof}
Suppose that $\phi\colon Z\til X$ is the $G$-equivariant continuous homogeneous  lifting of the canonical projection $\pi\colon Y\oplus Z\til Z$ given by Lemma \ref{lifting}. Then we may write
$$
\phi(z)=\begin{pmatrix} -\psi(z)  \\ z \end{pmatrix}
$$
for some continuous homogeneous $\psi\colon Z\til Y$. Now, for 
$$
T=\begin{pmatrix} u & \delta(u,v) \\ 0 & v \end{pmatrix}\in G,
$$
we have, by the $G$-equivariance  of $\phi$ and the identification $\dot T=v$, that $\phi v=T\phi$, i.e.,
$$
\begin{pmatrix} -\psi v(z)  \\ v(z) \end{pmatrix}
=\phi v(z)=T\phi(z)=\begin{pmatrix} u & \delta(u,v) \\ 0 & v \end{pmatrix}
\begin{pmatrix} -\psi(z)  \\ z \end{pmatrix}
=\begin{pmatrix} -u\psi(z)+\delta(u,v)(z)  \\v( z) \end{pmatrix}
$$
for all $z\in Z$. In other words, $-u\psi+\delta(u,v)=-\psi v$ or equivalently $\delta(u,v)=u\psi-\psi v$.
\end{proof}

Now, if $G\leqslant {\rm GL}(Y\oplus Z)$ a bounded subgroup leaving $Y$ invariant, we set
$$
U=\{u\in{\rm GL}(Y)\del \e v\in {\rm GL}(Z) \begin{pmatrix} u & \delta(u,v) \\ 0 & v \end{pmatrix}\in G\}
$$
and
$$
V=\{v\in{\rm GL}(Z)\del \e u\in {\rm GL}(Y) \begin{pmatrix} u & \delta(u,v) \\ 0 & v \end{pmatrix}\in G\}
$$
and note that these are bounded subgroups of ${\rm GL}(Y)$ and ${\rm GL}(Z)$ respectively.

Specialising Corollary \ref{isom general setting} to the setting above, we obtain the following.
\begin{thm}\label{theoreme10}
Let $X=Y\oplus Z$ be separable reflexive and $G\leqslant {\rm GL}(X)$ a bounded subgroup leaving $Y$ invariant. Assume that 
\begin{itemize}
\item[(i)] there are no closed linear $G$-invariant subspaces $\{0\}\varsubsetneq W\varsubsetneq Y$ nor superspaces $Y\varsubsetneq W\varsubsetneq X$, 
\item[(ii)] there is no closed linear $G$-invariant complement of $Y$ in $X$.
\end{itemize}
Then the mappings
$$
\begin{pmatrix} u & \delta(u,v) \\ 0 & v \end{pmatrix}\mapsto u
$$
and 
$$
\begin{pmatrix} u & \delta(u,v) \\ 0 & v \end{pmatrix}\mapsto v
$$
are topological group isomorphisms between $(G, \mathtt{sot})$ and $(U,\mathtt{sot})$, respectively $(G, \mathtt{sot})$ and $(V,\mathtt{sot})$.
\end{thm}


\section{Derivations and non-unitarisable representations}\label{hilbert case}
In this section, we apply our results from Section \ref{bounded} to the special case of Hilbert space, that is, we assume that $Y=Z=\ku H$, where $\ku H$ is the separable infinite-dimensional Hilbert space.  For this, we shall briefly review how to twist a unitary representation to obtain a non-unitarisable bounded representation.

So suppose that $\lambda\colon \Gamma\til \ku U(\ku H)$ is a unitary representation of a group $\Gamma$ on a separable infinite-dimensional Hilbert space $\ku H$.  A {\em derivation} associated to $\lambda$ is a map $d\colon \Gamma\til \ku B(\ku H)$, where $\ku B(\ku H)$ is the algebra of bounded linear operators on $\ku H$, satisfying the cocycle equation
$$
d(ab)=\lambda(a)d(b)+d(a)\lambda(b)
$$
for all $a,b\in \Gamma$.  Letting $\ku H_1$ and $\ku H_2$ denote two copies of $\ku H$,  this equation is simply equivalent to the requirement that the map $\lambda_{d}\colon \Gamma\til {\rm GL}(\ku H_1\oplus \ku H_2)$ given by
$$
\lambda_d(a)= \begin{pmatrix} \lambda(a) & d(a) \\ 0 & \lambda(a) \end{pmatrix}
$$
defines a representation, i.e., $\lambda_d(ab)=\lambda_d(a)\lambda_d(b)$. Moreover, this representation is bounded if and only if $d$ is {\em bounded}, i.e., $\sup_{a\in \Gamma}\norm{d(a)}<\infty$.

Now, as is well-known, the following statements are equivalent for a bounded derivation $d$ associated to $\lambda$. 
\begin{enumerate}
\item[(i)] $\lambda_{d}\colon \Gamma\til {\rm GL}(\ku H_1\oplus \ku H_2)$ is unitarisable,
\item[(ii)] $d$ is {\em inner}, that is, there is a bounded linear operator $L\in \ku B(\ku H)$ with $d(a)=\lambda(a)L-L\lambda(a)$,
\item[(iii)] there is a closed linear complement $\ku K$ of $\ku H_1$ in $\ku H_1\oplus \ku H_2$ invariant under the representation $\lambda_{d}\colon \Gamma\til {\rm GL}(\ku H_1\oplus \ku H_2)$.
\end{enumerate}

To see this, suppose first that $\ku K$ is a closed linear $\lambda_{d}$-invariant complement of $\ku H_1$ in $\ku H_1\oplus \ku H_2$ and let $P$ be the projection onto $\ku H_1$ along $\ku K$.  By the $\lambda_{d}$-invariance of $\ku K$, $P$ commutes with $\lambda_{d}(a)$ for all $a\in \Gamma$. So, viewing $d(a)$ as an operator from $\ku H_2$ to $\ku H_1$, for  all $x\in \ku H_2$, we have
$$
d(a)x+P\lambda(a)x=Pd(a)x+P\lambda(a)x=P\lambda_d(a)x= \lambda_d(a)Px=\lambda(a)Px,
$$
i.e., $d(a)x=  \lambda(a)Px-P\lambda(a)x$. Letting $L$ be the restriction of $P$ to $\ku H_2$, we thus find that $d(a)=\lambda(a)L-L\lambda(a)$ for all $a\in \Gamma$ and therefore $d$ is an inner derivation.

Also, if $d$ is inner and thus $d(a)=\lambda(a)L-L\lambda(a)$ for some bounded operator $L$, then, for all $a\in \Gamma$, we have
$$
\begin{pmatrix} \lambda(a) & d(a) \\ 0 & \lambda(a) \end{pmatrix} = 
\begin{pmatrix} \Id & -L \\ 0 & \Id \end{pmatrix}  
\begin{pmatrix} \lambda(a) & 0 \\ 0 & \lambda(a) \end{pmatrix}  
\begin{pmatrix} \Id & L \\ 0 & \Id \end{pmatrix}, 
$$
which shows that $\lambda_d\colon \Gamma\til {\rm GL}(\ku H_1\oplus \ku H_2)$ is similar to a block diagonal unitary representation and hence is unitarisable.

Finally, by the complete reducibility of unitary representations, if the representation $\lambda_d\colon \Gamma\til {\rm GL}(\ku H_1\oplus \ku H_2)$ is unitarisable, the $\lambda_d$-invariant subspace $\ku H_1$ has a $\lambda_d$-invariant  complement $\ku K$ in $\ku H_1\oplus \ku H_2$.

\begin{thm}\label{hilbert}
Suppose that $\lambda\colon \Gamma\til \ku U(\ku H)$ is an irreducible unitary representation of a group $\Gamma$ on a separable infinite-dimensional Hilbert space $\ku H$ and $d\colon \Gamma\til \ku B(\ku H)$ is an associated non-inner bounded derivation.
Let also $\ku H_1$ and $\ku H_2$ be distinct copies of $\ku H$ and suppose that $G \leqslant GL(\ku H_1 \oplus \ku H_2)$ is a bounded subgroup leaving $\ku H_1$ invariant and containing $\lambda_d[\Gamma]$. Then the mappings $G\til {\rm GL}(\ku H)$ defined by
$$
\begin{pmatrix} u & \delta(u,v) \\ 0 & v \end{pmatrix}\mapsto u
$$
and 
$$
\begin{pmatrix} u & \delta(u,v) \\ 0 & v \end{pmatrix}\mapsto v
$$
are topological group isomorphisms between $(G, \mathtt{sot})$ and $(U,\mathtt{sot})$, respectively $(G, \mathtt{sot})$ and $(V,\mathtt{sot})$.
\end{thm}

\begin{proof}  
First, by irreducibility of $\lambda$, there is no closed linear $\lambda_d$-invariant subspace $\{0\}\varsubsetneq \ku K\varsubsetneq \ku H_1$. Also, since $d$ is not inner, there is no $\lambda_d$-invariant closed linear complement of $\ku H_1$ in $\ku H_1\oplus \ku H_2$.

We also claim that  there is no closed linear $\lambda_d$-invariant superspace $\ku H_1\varsubsetneq\ku K\varsubsetneq \ku H_1\oplus \ku H_2$, as otherwise, $ \ku K\cap \ku H_2$ would be $\lambda$-invariant, contradicting the irreducibility of $\lambda$. Indeed, suppose that $x\in \ku K\cap \ku H_2$. Then $\lambda_d(a)x=d(a)x+\lambda(a)x\in \ku K$, whereby, as $d(a)x\in \ku H_1\subseteq \ku K$ and $\lambda(a)x\in \ku H_2$, also $\lambda(a)x\in \ku K\cap \ku H_2$.

Now, since $G$ contains $\lambda_d[\Gamma]$, it follows that there are no $G$-invariant subspaces of the above type. Therefore, $G$ satisfies the conditions of Theorem \ref{theoreme10}, whereby our result follows.
\end{proof}



\section{The representation of ${\rm Aut}(T)$ on $\ell_2(T)$}\label{Aut(T)}

In the following, we let $T$ denote the $\aleph_0$-regular tree, that is, the countable connected, symmetric, irreflexive graph without loops in which every vertex has infinite valence. One particular realisation of $T$ is as the Cayley graph of the free group on a denumerable set of generators, $\F_\infty$, with respect to its free generating set. We also let $\lambda$ denote the unitary representation of its automorphism group, $G={\rm Aut}(T)$, on the vector space $\C^T$ of $\C$-valued functions on $T$ given by
$$
\lambda(g)(x)=x(g\inv \,\cdot),
$$
for $g\in G$ and $x\in \C^T$, and note that the linear subspaces $\ell_p(T)$, $1\leqslant p\leqslant \infty$, and $c_0(T)\subseteq \C^T$ are $\lambda(G)$-invariant. The same holds for the space $c_{00}(T)$ of finitely supported functions. Let also $G_t={\rm Aut}(T,t)$ denote the isotropy subgroup  of the vertex $t\in T$ and, for a subset $A\subseteq T$, let $\C^A$ and $\ell_p(A)$ denote the subspaces of vectors whose support is included in $A$. We set $\1_t\in \C^T$ to be the Dirac function at the vertex $t\in T$.

We begin by two elementary observations that will be significantly strengthened later on.
\begin{prop}\label{irred}
The unitary representation $\lambda\colon G\til \ku U(\ell_2(T))$ is irreducible.
\end{prop}

\begin{proof}
Note that, since every vertex $s\neq t$ has infinite orbit in $T$ under the action of $G_t$,  $\C\1_t\subseteq \ell_2(T)$ is the $1$-dimensional subspace of $\lambda(G_t)$-invariant vectors. Now, if $\ell_2(T)=\ku H\oplus \ku H^\perp$ were a $\lambda(G)$-invariant decomposition of $\ell_2(T)$, the orthogonal projection $P$ onto $\ku H$ would commute with $\lambda(G)$ and so, in particular, $\lambda(g)P\1_t=P\lambda(g)\1_t=P\1_t$ for all $g\in G_t$, i.e., $P\1_t$ is $\lambda(G_t)$-invariant. It follows that $P\1_t\in \C\1_t\cap \ku H$ and so either  $\1_t\in \ku H^\perp$ or $\1_t\in \ku H$.
But, as $\lambda(G)\1_t$ spans $\ell_2(T)$, we see by the invariance of $\ku H^\perp$ and $\ku H$ that either $\ku H^\perp=\ell_2(T)$ or $\ku H=\ell_2(T)$, showing irreducibility.
\end{proof}

\begin{prop}\label{inner product}
$\lambda\colon G\til \ku U(\ell_2(T))$ is {\em uniquely unitarisable}, i.e., up to a scalar multiple, there is a unique $\lambda$-invariant inner product on $\ell_2(T)$ equivalent with the usual inner product.
\end{prop}

\begin{proof}
Fix $s\in T$ and enumerate the neighbours of $s$ in $T$ as $\{\ldots,t_1,t_0,t_1,\ldots\}$.  Pick a sequence of automorphisms $g_1,g_2,g_3,\ldots\in G_s$ so that $g_n(t_i)=t_{i+n}$ for all $n\geqslant 1$ and $i\in \Z$. Then $\lambda(g_n)\Lim{wot}P_s$, where $P_s$ denotes the usual orthogonal projection onto $\C\1_s$. 

Note  that $\ell_2(T\setminus \{s\})$ is a closed linear $\lambda(G_s)$-invariant complement of $\C\1_s$. On the other hand, if $\ku H\subseteq \ell_2(T)$ is any other closed linear  $\lambda(G_s)$-invariant complement of $\C\1_s$, then $P_sx=w\!-\!\lim \lambda(g_s)x$ belongs to $\C\1_s\cap \ku H=\{0\}$ for all $x\in \ku H$, whereby $\ku H\subseteq \ell_2(T\setminus \{s\})$ and hence $\ku H=\ell_2(T\setminus \{s\})$. It follows that $\ell_2(T\setminus \{s\})$ is the unique $\lambda(G_s)$-invariant closed linear complement of $\C\1_s$ in $\ell_2(T)$.

Now, suppose $\langle\cdot |\cdot \rangle$ denotes the usual inner product on $\ell_2(T)$ and $\langle\cdot |\cdot \rangle'$ is another $\lambda(G)$-invariant equivalent inner product on $\ell_2(T)$. Then, for every $s\in T$,  the orthogonal complement $(\C{\bf 1}_s)^{\perp'}$ is a closed linear $\lambda(G_s)$-invariant complement of $\C\1_s$, so $(\C{\bf 1}_s)^{\perp'}=\ell_2(T\setminus\{s\})=(\C{\bf 1}_s)^\perp$, whereby $\langle {\bf 1}_s| {\bf 1}_t\rangle'=0$ for all $s\neq t$ in $T$.

Since $\lambda(G)$ acts transitively on $\{ {\bf 1}_s\}_{s\in T}$, we also see that $\langle {\bf 1}_s|{\bf 1}_s\rangle'=\langle {\bf 1}_t|{\bf 1}_t\rangle'$ for all $s,t\in T$. So, up to multiplication by a scalar, we have  $\langle\cdot |\cdot \rangle=\langle\cdot |\cdot \rangle'$.
\end{proof}

We shall now significantly improve the preceding two results by removing any assumptions of continuity.

\begin{lemma}
For all $t\in T$ and $1\leqslant p<\infty$, there are no almost $\lambda(G_t)$-invariant unit vectors in $\ell_p(T\setminus\{t\})$.
\end{lemma}

\begin{proof}
Fix a countable non-amenable group $\Gamma$, a vertex $t\in T$ and enumerate the neighbours of $t$ in $T$ by the elements of $\Gamma$. Also, for every $a\in \Gamma$, let $T_a$ denote the subtree of all vertices $s\in T$ whose geodesic to $t$ passes through $a$. So the rooted trees $\{(T_a, a)\}_{a\in \Gamma}$ are all isomorphic to some fixed rooted tree $(T',r)$ and we can therefore identify $T\setminus \{t\}$ with $T'\times \Gamma$ in such a way that each  $T_a$ is identified with $T'\times \{a\}$ via the aforementioned isomorphism. Moreover, if we define an action $\rho\colon  \Gamma\curvearrowright \ell_p(T'\times \Gamma)$ by letting $\Gamma$ shift the second coordinate, it suffices to show that this action does not have almost invariant unit vectors.

To see this, note that, as $\Gamma$ is non-amenable,  the left regular representation $\sigma\colon \Gamma\curvearrowright \ell_p(\Gamma)$ does not have almost invariant unit vectors.  There are therefore  $g_1,\ldots,g_k\in \Gamma$ and $\eps>0$ so that 
\begin{equation}\label{equation1}
\max_{1\leqslant i\leqslant k}\;\norm{x-\sigma(g_i)x}>\eps\norm x
\end{equation}
for every non-zero vector $x\in \ell_p(\Gamma)$. For $s\in T'$, let $P_s\colon \ell_p(T'\times \Gamma)\til \ell_p(\{s\}\times \Gamma)$ denote the canonical projection and note that $P_s$ commutes with the $\rho(g_i)$. Now, fix $0\neq x\in \ell_p(T'\times \Gamma)$, set 
$$
N_i=\big\{s\in T'\del\norm{P_sx-P_s\rho(g_i)x}= \norm{P_sx-\rho(g_i)P_sx}>\eps\norm{P_sx}\big\}
$$
and note that, by (\ref{equation1}), $T'=\bigcup_{1\leqslant i\leqslant k}N_i$. We pick $i$ so that $\big(\sum_{s\in N_i}\norm{P_sx}^p\big)^\frac 1p\geqslant \frac1{k}\norm{x}$ and see that
$$
\norm{x-\rho(g_i)x}^p\geqslant \sum_{s\in N_i}\norm{P_sx-P_s\rho(g_i)x}^p>\sum_{s\in N_i} \eps^p\norm{P_sx}^p\geqslant \frac{\eps^p}{k^p}\norm{x}^p,
$$
i.e., $\norm{x-\rho(g_i)x}>\frac \eps k \norm{x}$. Thus, no unit vector in $\ell_p(T'\times \Gamma)$ is $\big(\rho(g_1),\ldots, \rho(g_k);\frac \eps k\big)$-invariant.
\end{proof}

The operator $R$ below occurs frequently in work on uniqueness of translation invariant functionals, e.g., \cite{bourgain}.

\begin{lemma}\label{commutant}
For all $t\in T$ and $1<p\leqslant \infty$, every linear operator $S\colon \ell_p(T)\til \C^T$ in the commutant of $\lambda(G_t)$ maps $\ell_p(T\setminus \{t\})$ into $\C^{T\setminus \{t\}}$. It follows that $\ell_p({T\setminus \{t\})}$ is the unique $\lambda(G_t)$-invariant linear complement of $\C\1_t$ in $\ell_p(T)$.
\end{lemma}

\begin{proof}Let $1\leqslant q<\infty$ be the conjugate index of $p$. 
Fix $g_1,\ldots, g_k\in G_t$ and $\eps>0$ so that 
$$
\max_{1\leqslant i\leqslant k}\;\norm{x-\lambda(g_i)x}>\eps\norm x
$$
for any non-zero vector $x\in \ell_q(T\setminus \{t\})$. It follows that the operator 
$$
R\colon \ell_q\big(T\setminus \{t\}\big)\;\;\longrightarrow\;\;\underbrace{ \ell_q\big(T\setminus \{t\}\big)\oplus\ldots \oplus  \ell_q\big(T\setminus \{t\}\big)}_{k\text{ copies}}
$$
 defined by $Rx=\big(x-\lambda(g_1)x,\ldots, x-\lambda(g_k)x\big)$ is an isomorphism with a closed subspace and therefore the conjugate operator $R^*$ is surjective. Thus, every element $x\in \ell_p\big(T\setminus \{t\}\big)$ can be written as 
 $$
 x=\sum_{i=1}^k y_i-\lambda(g_i)y_i,
 $$
 for some $y_i\in \ell_p\big(T\setminus \{t\}\big)$.

In particular, if $S\colon \ell_p\big(T\setminus \{t\}\big)\til \C^T$ is any linear operator commuting with $\lambda(G_t)$, then
$$
\1^*_t(Sx)=\sum_{i=1}^k \1_t^*Sy_i-\1_t^*\lambda(g_i)Sy_i=\sum_{i=1}^k \1_t^*Sy_i-\1_t^*Sy_i=0,
$$
as $\1_t^*\lambda(g_i)=\1^*_t$. 
That is, $S$ maps $\ell_p(T\setminus \{t\})$ into $\C^{T\setminus \{t\}}$.

Thus, if $P\colon \ell_p(T)\til \C\1_t$ is a linear projection commuting with $\lambda(G_t)$, then  $\ell_p(T\setminus \{t\})\subseteq \ker P$, and, since $\ell_p(T\setminus \{t\})$ is also a linear complement of $\C\1_t$, it follows that $\ell_p(T\setminus \{t\})=\ker P$, whereby $P$ is the projection along the subspace $\ell_p(T\setminus \{t\})$.
\end{proof}
 
We can now obtain the following strengthening of Lemma \ref{inner product}.
 \begin{thm}\label{unique inner product}
 The usual inner product is, up to a scalar multiple, the unique  $\lambda\big({\rm Aut}(T)\big)$-invariant inner product on $\ell_2(T)$.
 \end{thm}

\begin{proof}
Note that, if $\langle\cdot|\cdot\rangle'$ is a $\lambda(G)$-invariant inner product on $\ell_2(T)$, then, for every $t\in T$, the orthogonal complement $(\C\1_t)^{\perp'}$ of $\C\1_t$ with respect to $\langle\cdot|\cdot\rangle'$ is a $\lambda(G_t)$-invariant linear complement. So, by Lemma \ref{commutant},  we have $(\C\1_t)^{\perp'}=\ell_2(T\setminus \{t\})$ and, in particular, $\langle \1_s| \1_t\rangle'=0$ for all $s\neq t$ in $T$, whereby $\{\1_t\}_{t\in T}$ is an $\langle\cdot|\cdot\rangle'$-orthogonal sequence. Since $G$ acts transitively on $T$, we also see that $\langle\1_s|\1_s\rangle'=\langle\1_t|\1_t\rangle'>0$ for all $s,t\in T$ and hence, by multiplying by a positive scalar, we may suppose that $\{\1_t\}_{t\in T}$ is actually orthonormal with respect to $\langle\cdot|\cdot\rangle'$, whence the usual inner product agrees with $\langle\cdot|\cdot\rangle'$   on $c_{00}(T)$ .

Observe now that 
$$
c_{00}(T)^{\perp'}=\bigcap_{t\in T}\ell_2(T\setminus \{t\})=\{0\},
$$
showing that $c_{00}(T)$ is $\norm\cdot'$-dense in $\ell_2(T)$, where $\norm\cdot'$ is the norm induced by $\langle\cdot|\cdot\rangle'$. It then follows from Parseval's Equality applied to each of the inner products that any $x\in \ell_2(T)$ may be simultaneously approximated in the two norms by an element of $c_{00}(T)$, which, by Cauchy--Schwarz, implies that the two inner products agree on $\ell_2(T)$.
\end{proof}

\begin{thm}\label{trivial commutant}
The commutant of $\lambda\big({\rm Aut}(T)\big)$ in the space of linear operators from $\ell_p(T)$ to $\C^T$, $1<p\leqslant \infty$, is just $\C\!\cdot\! \Id$.
\end{thm}

\begin{proof}
Note that if $S$ belongs to the commutant, then, by Lemma \ref{commutant}, $S$ maps $\ell_p(T\setminus \{t\})$ into $\C^{T\setminus \{t\}}$ for every $t\in T$ and hence maps $\ell_p(A)=\bigcap_{t\notin A}\ell_p(T\setminus \{t\})$ into $\C^A=\bigcap_{t\notin A}\C^{T\setminus \{t\}}$ for all subsets $A\subseteq T$. So fix $t\in T$ and write $S\1_t=\alpha \1_t$ for some $\alpha\in \C$. Then, for any $g\in G$, we have
$$
S\1_{g(t)}=S\lambda(g)\1_t=\lambda(g)S\1_t=\alpha\lambda(g)\1_t=\alpha\1_{g(t)}.
$$
As $G$ acts transitively on $T$, it follows that $S\1_s=\alpha\1_s$ for all $s\in T$. Now, suppose that $x\in \ell_p(T)$ and $s\in T$ and write $x=y+\xi\1_s$, with $y\in \ell_p(T\setminus \{s\})$ and $\xi\in \C$. It then follows that
$$
\1_s^*(Sx)= \1_s^*(Sy)+\1_s^*\big(S(\xi \1_s)\big)=\alpha\xi,
$$
showing that $Sx=\alpha x$. Thus $S=\alpha\!\cdot\!\Id$.
\end{proof}

Let us also observe that Theorem \ref{trivial commutant} fails for $p=1$. Indeed, if we define $N\colon \ell_1(T)\til \ell_\infty(T)$ by $N(\1_s)=\sum_{t\in \ku N_s}\1_t$, where $\ku N_t$ is the set of neighbours of $s$ in $T$, then $N$ clearly commutes with every $\lambda(g)$, $g\in {\rm Aut}(T)$.

Theorems \ref{unique inner product} and \ref{trivial commutant} show strong rigidity properties of the representation $\lambda\colon {\rm Aut}(T)\til \ku U(\ell_2(T))$. In this conncetion, it is natural to ask whether, apart from determining the inner product, it also determines the norm on $\ell_2(T)$. Indeed, suppose $\lambda({\rm Aut}(T))\leqslant K\leqslant  GL(\ell_2(T))$ is a bounded subgroup. Then, by Proposition 2.3  \cite{furman}, there is  an equivalent $K$-invariant norm $\triple\cdot$ on $\ell_2(T)$ that is uniformly convex and uniformly smooth. Moreover, for any finite subtree $A\subseteq T$, the space $\ell_2(T)^{\lambda(G_A)}$ of $\lambda(G_A)$-invariant vectors is just $\ell_2(A)$. So, by the Alaoglu--Birkhoff Theorem (see, e.g., Thm 4.10 \cite{duke}), there is a projection $P_A$ of $\ell_2(T)$ onto the subspace $\ell_2(A)$ with $\triple{P_A}=1$. Furthermore, this must be the usual orthogonal projection  since it commutes with $\lambda(G_A)$. Note also that the same holds in the dual. Finally, by approximating by finite subtrees and passing to a wot-limit, one observes that $\triple{P_A}=1$ for all non-empty subtrees $A\subseteq T$. This puts serious restrictions on the norm $\triple\cdot$ and thus also on $K$.

\begin{prob}
Is every bounded subgroup $\lambda({\rm Aut}(T))\leqslant K\leqslant  GL(\ell_2(T))$  contained in $\ku U(\ell_2(T))$?
\end{prob}

Observe first that this is equivalent to asking whether every such $K$ is unitarisable, since then the $K$-invariant inner product must be the usual one and hence $K\leqslant \ku U(\ell_2(T))$.


\section{A derivation associated to ${\rm Aut}(T)$}
In the following, we shall study a well-known derivation giving rise to a non-unitarisable representation of $\F_\infty$ (see also \cite{pisier, ozawa} for different presentations). For this, we introduce a bounded linear operator
$$
L\colon \ell_1(T)\til \ell_1(T),
$$
where $T$ is the $\aleph_0$-regular tree as in Section \ref{Aut(T)}. We begin by fixing a root $e\in T$ and let $\hat{\cdot}\,\colon T\til T$ be the map defined by $\hat e=e$ and $\hat s=s_{n-1}$, whenever $s\neq e$ and $s_0, s_1, s_2,\ldots, s_{n-1}, s_n$ is the geodesic from $s_0=e$ to $s_n=s$. Also, for any $s\in T$, let $\ku N_s$ denote the set of neighbours of $s$ in $T$.

We then let $L\colon \ell_1(T)\til \ell_1(T)$  be the unique bounded linear operator satisfying
$$
L(\1_s)=\1_{\hat s}
$$
for $s\neq e$ and 
$$
L(\1_e)=0.
$$
Observe then that the adjoint operator $L^*\colon \ell_\infty(T)\til \ell_\infty(T)$ satisfies
$$
L^*(\1_s)=\Big(\sum_{t\in \ku N_s}\1_t\Big)-\1_{\hat s}=\sum_{\hat t=s}\1_t
$$
for $s\neq e$ and 
$$
L^*(\1_e)=\sum_{t\in \ku N_e}\1_t.
$$ 

In other words, if $N\colon \ell_1(T)\til \ell_\infty(T)$ is the bounded operator defined following Theorem \ref{trivial commutant} by $N(\1_s)=\sum_{t\in \ku N_s}\1_t$, then $L^*+L=N$, which commutes with every $\lambda(g)$, $g\in {\rm Aut}(T)$. From this it follows that, for every $g\in {\rm Aut}(T)$,
$$
\lambda(g)L-L\lambda(g)=L^*\lambda(g)-\lambda(g)L^*
$$ 
as operators on $\ell_1(T)$. We may hence conclude that
$$
d(g)=L^*\lambda(g)-\lambda(g)L^*
$$
defines an operator on $\ell_\infty(T)$ of norm at most $2$, which restricts to an operator on $\ell_1(T)$ of norm at most $2$ and therefore, by the Riesz-Thorin interpolation Theorem, that $d(g)$ restricts to an operator on $\ell_2(T)$ of norm at most $2$. As evidently
\[\begin{split}
d(gf)
&=L^*\lambda(gf)-\lambda(gf)L^*\\
&=(L^*\lambda(g)-\lambda(g)L^*)\lambda(f)+\lambda(g)(L^*\lambda(f)-\lambda(f)L^*)\\
&=d(g)\lambda(f)+\lambda(g)d(f),
\end{split}\]
we see that $d\colon {\rm Aut}(T)\til \ku B(\ell_2(T))$ is a bounded derivation.

Now, if one identifies $T$ with the Cayley graph of $\F_\infty$, it is known that even the restriction of $d$ to $\F_\infty$ viewed as translations of $T$ is non-inner (see, e.g., \cite{pisier,ozawa}). However, allowing for all of ${\rm Aut}(T)$, we see that not only is $d$ not defined by an element of $\ku B(\ell_2(T))$, but $L^*$ is essentially the only linear operator from $\ell_2(T)$ to $\C^T$ defining $d$.

\begin{thm}\label{L* perturbation}
Suppose $A\colon \ell_2(T)\til \C^T$ is a globally defined linear operator so that $d(g)=A\lambda(g)-\lambda(g)A$ for all $g\in {\rm Aut}(T)$. Then  $A=L^*+\vartheta\Id$ for some $\vartheta\in \C$. 
\end{thm}

\begin{proof}
Assume that $A\colon \ell_2(T)\til \C^T$ is as above. Then, for all $g\in {\rm Aut}(T)$,
$$
A\lambda(g)-\lambda(g)A=d(g)=L^*\lambda(g)-\lambda(g)L^*,
$$
i.e., $\big(A-L^*\big)\lambda(g)=\lambda(g)\big(A-L^*\big)$.
By Theorem \ref{trivial commutant}, it follows that $A-L^*=\vartheta{\rm Id}$ for some $\vartheta\in \C$ and our theorem follows.
\end{proof}

Thus, to see that $d$ is not inner or even that $d$ cannot be written as $d(g)=A\lambda(g)-\lambda(g)A$  with $A\colon \ell_2(T)\til \ell_2(T)$ a gobally defined linear operator, note that, in this case,  $A=L^*+\vartheta\Id$ for some $\vartheta$, whereby $L^*$ would have to map $\ell_2(T)$ into $\ell_2(T)$, which it does not.

However, even though the derivation $d\colon {\rm Aut}(T)\til \ku B(\ell_2(T))$ is not inner, by Lemma \ref{repofderivation}, we see that there is a continuous homogeneous map $\psi\colon \ell_2(T)\til \ell_2(T)$ so that
$$
d(g)=\lambda(g)\psi-\psi\lambda(g)
$$
for all $g\in {\rm Aut}(T)$.

In the following, we combine the results above with the analysis of Sections \ref{bounded} and  \ref{hilbert case}. So, to simplify notation, we let $\ku H_1$ and $\ku H_2$ denote two distinct copies of $\ell_2(T)$. Now, suppose that $G\leqslant {\rm GL}(\ku H_1\oplus \ku H_2)$ is a bounded subgroup leaving $\ku H_1$ invariant and containing $\lambda_d[{\rm Aut}(T)]$, i.e., containing the block matrices
$$
\begin{pmatrix} \lambda(g) & d(g) \\ 0 & \lambda(g) \end{pmatrix}
=\begin{pmatrix} \lambda(g) & L^*\lambda(g)-\lambda(g)L^* \\ 0 & \lambda(g)\end{pmatrix},
$$
for all $g\in {\rm Aut}(T)$. As we have seen in Section \ref{bounded}, there is a partial map
$$
\delta\colon {\rm GL}(\ku H_1)\times {\rm GL}(\ku H_2)\til \ku B(\ku H_2, \ku H_1)
$$
so that every element of $G$ is of the form
$$
\begin{pmatrix} u & \delta(u,v) \\ 0 & v \end{pmatrix}
$$
for some $u\in {\rm GL}(\ku H_1)$ and $v\in {\rm GL}(\ku H_2)$. Also, by Lemma \ref{repofderivation}, there is a continuous homogeneous map $\psi\colon \ku H_2\til \ku H_1$ so that
$$
\delta(u,v)=u\psi-\psi v
$$
for all $u,v$. 

Therefore, by the expressions for $d(g)$, we see that
$$
\lambda(g)\big(L^*+\psi\big)=\big(L^*+\psi\big)\lambda(g),
$$
when $L^*$ and $\psi$ are viewed as continuous maps $\ell_2(T)\til \ell_\infty(T)$, while
$$
\lambda(g)\big(L-\psi\big)=\big(L-\psi\big)\lambda(g),
$$
when $L$ and $\psi$ are viewed as continuous maps $\ell_1(T)\til \ell_2(T)$. In other words, $L^*+\psi$ and $L-\psi$ commute with $\lambda(g)$ for $g\in {\rm Aut}(T)$.

Now, for every subset $S\subseteq T$, let ${\rm Aut}(T)_S=\{g\in {\rm Aut}(T)\del g(t)=t, \; \a t\in S\}$ 
denote the pointwise stabiliser and note that $\lambda({\rm Aut}(T)_S)$ acts trivially on  $\ell_1(S)$. Since $L-\psi$ commutes with the $\lambda(g)$, we see that, for any $g\in {\rm Aut}(T)_S$ and $x\in \ell_1(S)$,
$$
(L-\psi) x=(L-\psi)\lambda(g)x=\lambda(g)(L-\psi)x,
$$
which means that $(L-\psi)x\in \ell_2(T)^{\lambda({\rm Aut}(T)_S)}$, where the latter denotes the subspace of $\lambda({\rm Aut}(T)_S)$-invariant vectors in $\ell_2(T)$. But, if $S$ is a finite subtree, then 
$$
\ell_2(T)^{\lambda({\rm Aut}(T)_S)}=\ell_2(S),
$$ 
showing that $L-\psi$ maps $\ell_1(S)$ into $\ell_2(S)$.  Approximating arbitrary subtrees  by finite subtrees and extending by continuity, we conclude that $L-\psi$ maps $\ell_1(S)$ into $\ell_2(S)$ for all subtrees $S\subseteq T$. However, $L$ maps $\ell_1(S)$ into $\ell_1(S\cup \hat S)$, which shows that $\psi$ maps $\ell_1(S)$ into $\ell_2(S\cup \hat S)$. More precisely, assuming that $e\notin S$, if $s\in S$ denotes the vertex closest to $e$, we have, for all $x\in \ell_1(S)$,
$$
\1_{\hat s}^*(Lx)=\1_s(x),
$$
and so $\1_{\hat s}^*(\psi x)=\1_s(x)$.

Observe also that the continuous homogenous map $\Delta\colon \ku H_2\times \ku H_2\til \ku H_1$ given by 
$$
\Delta(x,y)=\psi(x)+\psi(y)-\psi(x+y)
$$
satisfies $\Delta(x,y)=(\psi-L)x+(\psi-L)y-(\psi-L)(x+y)$ for $x,y\in \ell_1(T)$, as $L$ is linear. Therefore, for any subtree $S\subseteq T$, $\Delta$ maps $\ell_1(S)\times \ell_1(S)$ into $\ell_2(S)$ and hence, by density of $\ell_1(S)$ in $\ell_2(S)$, 
$$
\Delta\colon \ell_2(S)\times \ell_2(S)\til \ell_2(S).
$$
Also, as $\delta(u,v)=u\psi-\psi v$ is linear, one readily verifies that $\Delta(vx,vy)=u\Delta(x,y)$ for all $x,y$.

Finally, since by Proposition \ref{irred} the unitary representation $\lambda\colon {\rm Aut}(T)\til \ku U(\ell_2(T))$ is irreducible, it follows from Theorem \ref{hilbert} that the maps
$$
\begin{pmatrix} u & \delta(u,v) \\ 0 & v \end{pmatrix}\mapsto u
$$
and 
$$
\begin{pmatrix} u & \delta(u,v) \\ 0 & v \end{pmatrix}\mapsto v
$$
are $\tt{sot}$-isomorphisms between $G$ and the respective images in ${\rm GL}(\ku H_1)$ and ${\rm GL}(\ku H_2)$.

We sum up the above discussion in the following theorem.

\begin{thm}\label{structure}
Suppose that $G\leqslant {\rm GL}(\ell_2(T)\oplus\ell_2(T))$ is a bounded subgroup leaving  the first copy of $\ell_2(T)$ invariant and containing $\lambda_d[{\rm Aut}(T)]$.

Then there is a continuous homogeneous map $\psi\colon \ell_2(T)\til \ell_2(T)$ for which
$$
L^*+\psi\colon \ell_2(T)\til \ell_\infty(T)\quad \text{and}\quad  L-\psi\colon \ell_1(T)\til \ell_2(T)
$$
commute with $\lambda(g)$ for $g\in {\rm Aut}(T)$ and so that every element of $G$ is of the form
$$
\begin{pmatrix} u & u\psi-\psi v \\ 0 & v \end{pmatrix}
$$
for some $u,v\in {\rm GL}(\ell_2(T))$. 

Finally, the mappings 
$$
\begin{pmatrix} u & u\psi-\psi v \\ 0 & v \end{pmatrix}\mapsto u
\qquad \text{and}\qquad
\begin{pmatrix} u &u\psi-\psi v \\ 0 & v \end{pmatrix}\mapsto v
$$
are $\tt{sot}$-isomorphisms between $G$ and their respective images in ${\rm GL}(\ell_2(T))$.
\end{thm}

Using the information given by Theorem \ref{structure} and its proof, one may compute some simple values of the function $\psi$ associated to a bounded subgroup $G$.

\begin{exa}Since $\psi-L$ maps $\ell_1(\{e\})=\C\1_e$ into $\ell_2(\{e\})=\C\1_e$ and $L\1_e=0$, we must have $\psi(\1_e)=\mu \1_e$ for some $\mu \in \C$. Thus, for any other vertex $s\in T\setminus \{e\}$, write $s=g(e)$ for some $g\in {\rm Aut}(T)$, whereby
$$
\psi(\1_s)=L\1_s-(L-\psi)\lambda(g)\1_e=\1_{\hat s}-\lambda(g)(L-\psi)\1_e=\1_{\hat s}-\mu \1_s.
$$
\end{exa}

\begin{exa}
Suppose that $s\neq e$ and $\hat s=e$. Then $(\psi-L)(\1_s+\1_e)=\mu\1_s+\nu\1_e$ for some $\mu,\nu\in \C$, whence $\psi(\1_s+\1_e)=\mu\1_s+(1+\nu)\1_e$ with $\mu, \nu$ independent of $s$. Again, for any pair $t$ and $\hat t$ of neighbouring vertices in $T\setminus \{e\}$, there is a $g\in {\rm Aut}(T)$  so that $g(s)=t$ and $g(e)=\hat t$, whereby
\[\begin{split}
\psi(\1_{t}+\1_{\hat t})=(\psi-L)(\1_{t}+\1_{\hat t})+L(\1_{t}+\1_{\hat t})=\mu\1_{t}+(1+\nu)\1_{\hat t}+\1_{\hat{\hat t}}.
\end{split}\]
\end{exa}

\begin{exa}
Consider now the special case when the map $\delta$ is defined by $\delta(u,v)=Au-vA$ for some globally defined linear operator $A\colon \ell_2(T)\til \C^T$ (note that this requires the $v$ to be defined from $A[\ell_2(T)]$ into $\C^T$). Then, by Theorem \ref{L* perturbation}, we have that $A=L^*+\vartheta\Id$ for some $\vartheta\in\C$, i.e., 
$$
\delta(u,v)=L^*u-vL^*+\vartheta(u-v).
$$
Moreover, as
$$
\begin{pmatrix}\Id  & \vartheta \Id  \\  0&  \Id\end{pmatrix}
\begin{pmatrix}u  & L^*u-vL^*+\vartheta(u-v) \\0  &  v\end{pmatrix}
\begin{pmatrix}\Id  & -\vartheta \Id  \\ 0 &  \Id\end{pmatrix}
=\begin{pmatrix}u  & L^*u-vL^* \\ 0 &  v\end{pmatrix},
$$
we see that by conjugating $G$ by the bounded operator $\begin{pmatrix}\Id  & \vartheta \Id  \\ 0 &  \Id\end{pmatrix}$, we obtain another bounded subgroup $G'$ with a corresponding map $\delta'(u,v)=L^*u-vL^*$.
\end{exa}


\begin{thebibliography}{99}
\bibitem{furman}U. Bader, A. Furman, T. Gelander and N. Monod, {\em Property (T) and rigidity for actions on Banach spaces}, Acta Math. 198 (2007), no. 1, 57--105.

\bibitem{Banach} S. Banach, {\em Th\'eorie des op\'erations lin\'eaires.} (French) [Theory of linear operators] Reprint of the 1932 original. \'Editions Jacques Gabay, Sceaux, 1993.

\bibitem{becerra}J. Becerra Guerrero and  A. Rodr\'iguez-Palacios, {\em  Transitivity of the Norm on Banach Spaces}, Extracta Mathematicae Vol. 17 (2002) no.  1, 1-- 58.

\bibitem{bourgain}J. Bourgain, {\em Translation invariant forms on $L^p(G)$,  ($1<p<\infty$)},  Ann. Inst. Fourier (Grenoble) 36 (1986), no. 1, 97--104.


\bibitem{CS} F. Cabello-S\'anchez, {\em Regards sur le probl\`eme des rotations de Mazur}, Extracta Math. {\bf 12} (1997), 97--116.

\bibitem{day}M. Day, {\em Means for the bounded functions and ergodicity of the bounded representations of semi-groups},  Trans. Amer. Math. Soc. 69 (1950), 276--291.

\bibitem{DGZ} R. Deville, G. Godefroy and V. Zizler, {\em Smoothness and renormings in Banach spaces},  Pitman Monographs and Surveys in Pure and Applied Mathematics, 64. Longman Scientific \& Technical, Harlow; copublished in the United States with John Wiley \& Sons, Inc., New York, 1993.

\bibitem{beata}S. J. Dilworth and B. Randrianantoanina, {\em Almost transitive and maximal norms in classical Banach spaces}, preprint.


\bibitem{dixmier}J. Dixmier, {\em  Les moyennes invariantes dans les semi-groupes et leurs applications}, 
Acta Sci. Math. Szeged 12 (1950), 213--227.

\bibitem{mautner}L. Ehrenpreis and F. I. Mautner, {\em Uniformly bounded representations of groups},
Proc. Nat. Acad. Sci. U. S. A. 41 (1955), 231--233.


\bibitem{fabian}M. Fabian, P. Habala, P. H\'ajek, V. Montesinos and V.  Zizler, {\em  Banach space theory. The basis for linear and nonlinear analysis},  CMS Books in Mathematics/Ouvrages de Math\'ematiques de la SMC. Springer, New York, 2011.

\bibitem{duke}V. Ferenczi and C. Rosendal, {\em On isometry groups and maximal symmetry}, Duke Math. J. 162 (2013), no. 10, 1771--1831.

\bibitem{lancien} G. Lancien, {\em Dentability indices and locally uniformly convex renormings}, Rocky Mountain Journal of Mathematics {\bf 23} (1993), no. 2, 635--647.

\bibitem{lusky} W. Lusky, {\em A note on rotations in separable Banach spaces}, Studia Math. {\bf 65} (1979), 239--242.

\bibitem{mazur} S. Mazur, {\em Quelques propri\'et\'es caract\'eristiques des espaces euclidiens}, C. R. Acad. Sci. Paris {\bf 207} (1938), 761--764.

\bibitem{ozawa}N. Ozawa, {\em   An Invitation to the Similarity Problems (after Pisier)}, Surikaisekikenkyusho Kokyuroku, 1486 (2006), 27--40.

\bibitem{pelczynski} A. Pe\l czy\'nski and S. Rolewicz, {\em Best norms with respect to isometry groups in normed linear spaces}, Short communication on International Mathematical Congress in Stockholm (1964), 104.

\bibitem{pisier}G. Pisier, {\em Similarity problems and completely bounded maps. Second, expanded edition. Includes the solution to ``The Halmos problem''},  Lecture Notes in Mathematics, 1618. Springer-Verlag, Berlin, 2001.

\bibitem{pytlic}T. Pytlic and R. Szwarc, {\em An analytic family of uniformly bounded representations of free groups},  
Acta Math. 157 (1986), no. 3-4, 287--309.



\bibitem{rolewicz} S. Rolewicz, {\em Metric linear spaces}, second. ed., Polish Scientific Publishers, Warszawa, 1984.



\end{thebibliography}
\end{document}